\documentclass{amsart}

\newtheorem{theorem}{Theorem}[section]
\newtheorem{lemma}[theorem]{Lemma}
\newtheorem{corollary}[theorem]{Corollary}
\newtheorem{remark}[theorem]{Remark}

\theoremstyle{definition}

\numberwithin{equation}{section}

\begin{document}

\title{Partial sums of the hyperharmonic series}

\author[author1]{Hongguang Wu}
\address{Department of Mathematics, Changzhou University, Changzhou 213164, China}
\email{whg@cczu.edu.cn}

\author[author2]{Jun Qiu}
\address{Department of Mathematics and Computer Science, Tongling University, Tongling, 244000, China}
\email{woshiqiujun@live.com}

\date{\today }



\begin{abstract}
	In \cite{ref2}, Erd\"{o}s and Niven proved that no two partial sums of the harmonic series are equal. In this paper, we extend this result by demonstrating that no two partial sums of the hyperharmonic series are equal
\end{abstract}

\maketitle

\section{Introduction}
In mathematical analysis, the harmonic series is a classical example of a divergent infinite series, defined as the sum of the reciprocals of the natural numbers:  
$$ H = \sum_{k=1}^{\infty} \frac{1}{k}. $$
Despite the monotonic decrease of its terms, the harmonic series diverges as the number of terms increases, illustrating that term-wise decay alone does not ensure convergence.

A generalized form of the harmonic series is given by  
$$ \sum_{n=1}^{\infty} \frac{1}{an + b}, $$  
where \( a \) and \( b \) are real numbers such that \( \frac{b}{a} \) is a positive integer.  
Another well-known generalization is the $p$-series (also referred to as the hyperharmonic series), defined as  
$$ \sum_{n=1}^{\infty} \frac{1}{n^p}, $$  
where \( p \) is a real number. When \( p = 1 \), the $p$-series reduces to the classical harmonic series.

In 1915, Theisinger \cite{LT} proved that the \( n \)-th harmonic sum  
$$1 + \frac{1}{2} + \cdots + \frac{1}{n}$$is never an integer for any \( n > 1 \).

In 1923, Nagell \cite{TN} extended Theisinger’s result to general arithmetic progressions, showing that the reciprocal sum  
$$\sum\limits_{i=0}^{n-1} \frac{1}{a + bi}$$  
is never an integer for \( n \geq 2 \), where \( a \) and \( b \) are positive integers.

Over two decades later, Erd\"{o}s and Niven \cite{ref2} further generalized Nagell’s theorem by proving that none of the elementary symmetric functions of  
$$\frac{1}{a}, \frac{1}{a + b}, \ldots, \frac{1}{a + (n-1)b}$$
can be an integer, under the same conditions.

Building upon the result of Erd\"{o}s and Niven, several researchers have obtained further significant developments in recent years. In particular, Chen and Tang \cite{YC} demonstrated that for \( n \geq 4 \), none of the elementary symmetric functions of the sequence $$1, \frac{1}{2}, \ldots, \frac{1}{n} $$is an integer.

In \cite{CW}, Wang and Hong investigated the case where \( n \geq 2 \) and proved that none of the elementary symmetric functions of the sequence $$1, \frac{1}{3}, \ldots, \frac{1}{2n - 1}$$ is an integer. In a subsequent work \cite{CW2}, they further showed that for positive integers \( a, b, n, k \) with \( 1 \leq k \leq n \), the quantity  
\[
S_{a,b}(n, k) = \sum_{0 \leq i_1 \leq \cdots \leq i_k \leq n-1} \prod_{j=1}^{k} \frac{1}{a i_j + b}
\]  
is not an integer, except in the two exceptional cases: \( b = n = k = 1 \), or \( a = b = 1 \), \( n = 3 \), and \( k = 2 \), where \( S_{a,b}(n, k) \) attains an integer value.  
Feng et al. \cite{YL} further extended Nagell’s theorem by proving that for \( n \geq 2 \) and any sequence of positive integers \( s_1, \ldots, s_n \), the reciprocal power sum  
\[
\sum_{k=1}^{n} \frac{1}{(a + (k - 1)b)^{s_k}}
\]  
is never an integer.

In \cite{ref2}, Erd\"{o}s and Niven also proved that no two partial sums of the harmonic series can be equal; that is, it is impossible for  
\[
\frac{1}{m} + \frac{1}{m+1} + \cdots + \frac{1}{n} = \frac{1}{x} + \frac{1}{x+1} + \cdots + \frac{1}{y}
\]  
to hold for distinct intervals.  
 
In this paper, we extend the result of Erd\"{o}s and Niven to the 2-series. Our main theorem establishes the analogous non-equivalence of partial sums in this generalized setting.

\begin{theorem}\label{1.2}(Theorem \ref{mainTH})
	No two partial sums of the 2-series are equal; that is, there do not exist integers \( m \leq n \) and \( x \leq y \), with \( (m, n) \neq (x, y) \), such that
	\[
	\frac{1}{m^2} + \frac{1}{(m+1)^2} + \cdots + \frac{1}{n^2} = \frac{1}{x^2} + \frac{1}{(x+1)^2} + \cdots + \frac{1}{y^2}.
	\]
\end{theorem}

\section{Preliminary Lemmas}
In this section, we present several lemmas that are essential for the proof of Theorem~\ref{1.2}.
\vspace{0.1cm}

\begin{lemma}[\cite{ref3}]\label{r3}
	For any positive integer $n$, there is always a prime in the interval $[n,2n].$
\end{lemma}
\begin{remark}
	For any integer \( n > 1 \), there exists at least one prime number in the interval \( [n, 2n - 1] \).
\end{remark}

\begin{lemma}[\cite{ref3}]\label{1}
	If $n>k$, then in the set $\{n, n+1,\cdots,n+k-1\}$ there is an
	integer containing a prime divisor greater than or equal to $k+1$.
\end{lemma}
\begin{theorem}[\cite{HW}]\label{t}
	Let $a,b,k$ be positive integers, and $(a,b)=1$, we have
	\begin{align*}
	\mathrm{lcm}\{a,a + b,...,a + nb\}\ge\prod\limits_{p\mid b}p^{v_{p}(n!)}\frac{1}{n!}\prod\limits_{i=0}^{n}(a+ib).
	\end{align*}
\end{theorem}
\begin{lemma}\label{9}
	If \( n \geq (k+1)^2 \), then the set \( \{n, n+1, \ldots, n+k\} \) contains at least one integer that has a prime divisor greater than or equal to \( 2(k+1) \).
\end{lemma}

\begin{proof}
	By Lemma~\ref{t}, we have
	\[
	\mathrm{lcm}\{n, n+1, \ldots, n+k\} \geq \frac{1}{k!} \prod_{i=0}^{k} (n+i).
	\]
	Let \( p \) be any prime divisor of \( \mathrm{lcm}\{n, n+1, \ldots, n+k\} \). Then there exists some \( 0 \leq i_p \leq k \) such that
	\[
	p^{v_p(\mathrm{lcm}\{n, \ldots, n+k\})} = p^{v_p(n + i_p)} \leq n + i_p.
	\]
	Define the set
	\[
	A := \left\{ p \in \mathbb{N} : p \mid \mathrm{lcm}\{n, \ldots, n+k\}, \ p \text{ is prime} \right\}.
	\]
	Consider the product of the prime powers corresponding to the smallest \( \frac{k}{2} + 1 \) primes in \( A \). We then obtain
	\[
	\frac{ \mathrm{lcm}\{n, \ldots, n+k\} }{ \prod\limits_{i=0,\, p_i \in A}^{\lfloor\frac{k}{2}+1\rfloor} p_i^{v_{p_i}(\mathrm{lcm}\{n, \ldots, n+k\})} } 
	\geq \frac{1}{k!} \prod_{i=0}^{\lceil\frac{k}{2}-1\rceil} (n+i)
	\geq \frac{(k+1)^k}{k!} > 1.
	\]
	Therefore, the number of distinct prime divisors in \( A \) must be at least \( \frac{k}{2} + 2 \), implying that at least one of them satisfies \( p \geq 2(k+1) \).
\end{proof}

The following lemma follows directly from standard results and is straightforward to verify.
\begin{lemma}\label{2}
	Let \( r \in \mathbb{N} \).
	
	\begin{itemize}
		\item[(i)] If \( r \) is even, then
		\begin{align*}
			\sum_{i=1}^{\frac{r}{2}} i^2 &= \frac{(r+1)^3 - (r+1)}{24}, \\
			\sum_{i=1}^{\frac{r}{2}} i^4 &= \frac{3(r+1)^5 - 10(r+1)^3 + 7(r+1)}{480}, \\
			\sum_{i=1}^{\frac{r}{2}} i^6 &= \frac{3(r+1)^7 - 21(r+1)^5 + 49(r+1)^3 - 31(r+1)}{2688}.
		\end{align*}
		
		\item[(ii)] If \( r \) is odd, then
		\begin{align*}
			\sum_{i=1}^{\frac{r+1}{2}} (2i - 1)^2 &= \frac{(r+1)^3 - (r+1)}{6}, \\
			\sum_{i=1}^{\frac{r+1}{2}} (2i - 1)^4 &= \frac{3(r+1)^5 - 10(r+1)^3 + 7(r+1)}{30}, \\
			\sum_{i=1}^{\frac{r+1}{2}} (2i - 1)^6 &= \frac{3(r+1)^7 - 21(r+1)^5 + 49(r+1)^3 - 31(r+1)}{42}.
		\end{align*}
	\end{itemize}
\end{lemma}

\section{Proof of Theorem \ref{1.2}}
In this section, we provide a detailed proof of Theorem \ref{1.2}.
Let \( G(a, r) := \sum\limits_{i=0}^{r} \frac{1}{(a+i)^2} \) for \( a, r \in \mathbb{N}^* \).

\begin{lemma}\label{lem1}
	Let \( a_1 < a_2 \leq a_1 + r \) and \( a_1, a_2, r, s \in \mathbb{N}^* \). If \( G(a_1, r) = G(a_2, s) \), then
	\[
	G(a_1, a_2 - a_1 - 1) = G(a_1 + r + 1, a_2 + s - a_1 - r - 1).
	\]
\end{lemma}

\begin{proof}
	Since \( a_1 < a_2 \leq a_1 + r \), it follows that \( s > r \). Observe that
	\[
	G(a_1, r) = \sum_{i = a_1}^{a_2 - 1} \frac{1}{i^2} + \sum_{i = a_2}^{a_1 + r} \frac{1}{i^2},
	\]
	and
	\[
	G(a_2, s) = \sum_{i = a_2}^{a_1 + r} \frac{1}{i^2} + \sum_{i = a_1 + r + 1}^{a_2 + s} \frac{1}{i^2}.
	\]
	Equating the two expressions yields
	\[
	\sum_{i = a_1}^{a_2 - 1} \frac{1}{i^2} = \sum_{i = a_1 + r + 1}^{a_2 + s} \frac{1}{i^2},
	\]
	which is equivalent to
	\[
	G(a_1, a_2 - a_1 - 1) = G(a_1 + r + 1, a_2 + s - a_1 - r - 1),
	\]
	as desired.
\end{proof}

For any positive integer \( n \), define
\[
\varepsilon_n := \frac{2n + 1 - \sqrt{4n^2 + 1}}{2}.
\]
It is straightforward to verify that the following identity holds:
\[
\frac{1}{n - \varepsilon_n} - \frac{1}{n + 1 - \varepsilon_n} = \frac{1}{n^2}.
\]
Moreover, by the monotonicity of \( \varepsilon_n \), we have
\[
0 < \varepsilon_a < \varepsilon_{a+1} < \cdots < \varepsilon_{a + r} < \frac{1}{2}.
\]
	
\begin{lemma}
	Let \( a, r \in \mathbb{N}^{*} \). Then there exists a real number \( \eta \) satisfying
	\[
	\varepsilon_a < \eta < \varepsilon_{a+r}
	\]
	such that
	\[
	\sum_{i=0}^{r} \frac{1}{(a+i)^2} = \frac{r+1}{(a+r+1 - \eta)(a - \eta)}.
	\]
\end{lemma}

\begin{proof}
	For each \( i \geq 0 \), we have
	\[
	\frac{1}{(a+i)^2} = \frac{1}{a+i - \varepsilon_{a+i}} - \frac{1}{a+i+1 - \varepsilon_{a+i}}.
	\]
	Note that
	\[
	\frac{1}{a+i - \varepsilon_{a+i}} - \frac{1}{a+i+1 - \varepsilon_{a+i}} = \frac{1}{(a+i)^2 + (1 - 2\varepsilon_{a+i})(a+i) + \varepsilon_{a+i}(\varepsilon_{a+i} - 1)}.
	\]
	Since \( \varepsilon_n \) is increasing and less than \( \tfrac{1}{2} \), we have \( \varepsilon_{a+i} \geq \varepsilon_a \), and hence
	\[
	\frac{1}{(a+i)^2} \geq \frac{1}{a+i - \varepsilon_a} - \frac{1}{a+i+1 - \varepsilon_a}.
	\]
	Summing both sides from \( i = 0 \) to \( r \), we get
	\begin{align}
		\sum_{i=0}^{r} \frac{1}{(a+i)^2} > \frac{1}{a - \varepsilon_a} - \frac{1}{a + r + 1 - \varepsilon_a} = \frac{r+1}{(a + r + 1 - \varepsilon_a)(a - \varepsilon_a)}.
		\label{3.1}
	\end{align}

	Similarly, using \( \varepsilon_{a+i} \leq \varepsilon_{a+r} \), we obtain
	\begin{align}
	\sum_{i=0}^{r} \frac{1}{(a+i)^2} < \frac{r+1}{(a + r + 1 - \varepsilon_{a+r})(a - \varepsilon_{a+r})}.
	\label{3.2}
\end{align}
	
	Combining inequalities \eqref{3.1} and \eqref{3.2}, by the intermediate value theorem, there exists \( \eta \in (\varepsilon_a, \varepsilon_{a+r}) \) such that
	\[
	\sum_{i=0}^{r} \frac{1}{(a+i)^2} = \frac{r+1}{(a+r+1 - \eta)(a - \eta)}.
	\]
\end{proof}

\begin{remark}
	There exists a real number \( \widetilde{r} \in (0, r) \) such that
	\[
	\eta = \frac{2(a + \widetilde{r}) + 1 - \sqrt{4(a + \widetilde{r})^{2} + 1}}{2}.
	\]
	Consequently, we have the following inequality:
	\begin{align}\label{q1}
		1 - 2\eta = \frac{\sqrt{4(a + \widetilde{r})^2 + 1} + 1 - 2(a + \widetilde{r})}{2(a + \widetilde{r}) + \sqrt{4(a + \widetilde{r})^2 + 1}} 
		> \frac{1}{4(a + \widetilde{r}) + 1},
	\end{align}
	and similarly,
	\begin{align}\label{q2}
		1 - 2\eta = \frac{\sqrt{4(a + \widetilde{r})^2 + 1} + 1 - 2(a + \widetilde{r})}{2(a + \widetilde{r}) + \sqrt{4(a + \widetilde{r})^2 + 1}} 
		< \frac{2}{4(a + \widetilde{r}) + 1}.
	\end{align}
\end{remark}

\begin{remark}\label{re3.4}
	If \( G(a_1, r) = G(a_2, s) \), we may assume without loss of generality that \( a_1 < a_2 \). Moreover, by Lemma~\ref{lem1}, we can further restrict to the case \( a_1 + r < a_2 \).
\end{remark}
	 
From this point onward, we assume without further comment that \( a_1 + r < a_2 \).
 \begin{lemma}
 	Let
 	\[
 	G(a_1, r) = \frac{r+1}{(a_1 + r + 1 - \eta_1)(a_1 - \eta_1)}, \quad
 	G(a_2, s) = \frac{r+1}{(a_2 + s + 1 - \eta_2)(a_2 - \eta_2)},
 	\]
 	where
 	\[
 	\varepsilon_{a_1} < \eta_1 < \varepsilon_{a_1 + r}, \quad
 	\varepsilon_{a_2} < \eta_2 < \varepsilon_{a_2 + s}.
 	\]
 	
 	If \( G(a_1, r) = G(a_2, s) \), then the following inequalities hold:
 	\begin{align}
 		\left| (s+1)\left[(4a_1 + 2r)(1 - 2\eta_1) - 1 + (1 - 2\eta_1)^2 \right] \right| &< 1, \label{eq8} \\
 		\left| (r+1)\left[(4a_2 + 2s)(1 - 2\eta_2) - 1 + (1 - 2\eta_2)^2 \right] \right| &< 1. \label{eq8-2}
 	\end{align}
 \end{lemma}
\begin{proof}We prove inequality~\eqref{eq8}; the proof of \eqref{eq8-2} is analogous and omitted.
	
	To begin, we observe the following identity:
	\begin{align}\label{eq7}
		&(r+1)\left[(2a_2 - 1)(2a_2 + 2s + 1) + 1\right]
		- (s+1)\left[(2a_1 - 1)(2a_1 + 2r + 1) + 1\right] \nonumber \\
		=\, &(s+1)\left[(4a_1 + 2r)(1 - 2\eta_1) - 1 + (1 - 2\eta_1)^2\right] \nonumber \\
		& - (r+1)\left[(4a_2 + 2s)(1 - 2\eta_2) - 1 + (1 - 2\eta_2)^2\right].
	\end{align}
	
	Next, we estimate the left-hand term in~\eqref{eq8} and establish that
	\begin{equation}\label{eq9}
		\left|(s+1)\left[(4a_1 + 2r)(1 - 2\eta_1) - 1 + (1 - 2\eta_1)^2\right]\right|
		< \frac{(s+1)^2 a_1}{2(a_2 - 1)^2}.
	\end{equation}
	We observe the following inequality:
	\begin{align*}
		&(4a_1 + 2r)(1 - 2\eta_1) - 1 + (1 - 2\eta_1)^2 \\
		>& \frac{4a_1 + 2r}{4(a_1 + \widetilde{r}) + 1} - 1 + \left( \frac{1}{4(a_1 + \widetilde{r}) + 1} \right)^2 \\
		>& \frac{4a_1 + 2r}{4(a_1 + r) + 1} - 1 + \left( \frac{1}{4(a_1 + \widetilde{r}) + 1} \right)^2 \\
		=& \frac{-2r - 1}{4(a_1 + r) + 1} + \left( \frac{1}{4(a_1 + \widetilde{r}) + 1} \right)^2 \\
		>& \frac{-2r - 1}{4(a_1 + r)}.
	\end{align*}
	Hence, we obtain the lower bound:
	\begin{equation}\label{3.9}
		(4a_1 + 2r)(1 - 2\eta_1) - 1 + (1 - 2\eta_1)^2 > \frac{-2r - 1}{4(a_1 + r)}.
	\end{equation}
	
	Similarly, using the fact that \( \eta_1 < \varepsilon_{a_1 + r} < \tfrac{1}{2} \), we also have the upper bound:
	\begin{equation}\label{3.10}
		(4a_1 + 2r)(1 - 2\eta_1) - 1 + (1 - 2\eta_1)^2 < \frac{2r + 1}{4(a_1 + r)}.
	\end{equation}
	
	Combining inequalities~\eqref{3.9} and~\eqref{3.10}, we deduce that
	\[
	\left|(s+1)\left[(4a_1 + 2r)(1 - 2\eta_1) - 1 + (1 - 2\eta_1)^2\right]\right|
	< \frac{(2r + 1)(s + 1)}{4(a_1 + r)}.
	\]
By inequality~\eqref{q2}, we observe that
\begin{align*}
	\frac{(a_1 - \eta_1)^2}{r+1}
	&= \frac{(a_2 - \eta_2)^2}{s+1} + (a_2 - a_1) - (\eta_2 - \eta_1) \\
	&= \frac{(a_2 - \eta_2)^2}{s+1} + a_2 - a_1 + \left( \frac{1}{2} - \eta_2 \right) - \left( \frac{1}{2} - \eta_1 \right) \\
	&> \frac{(a_2 - \eta_2)^2}{s+1}.
\end{align*}

Since \( \eta_1 > 0 \), we have \( a_1 > a_1 - \eta_1 \), and thus
\[
\frac{a_1^2}{r+1} > \frac{(a_1 - \eta_1)^2}{r+1} > \frac{(a_2 - \eta_2)^2}{s+1} > \frac{(a_2 - 1)^2}{s+1}.
\]
It follows that
\[
\frac{r+1}{a_1^2} < \frac{s+1}{(a_2 - 1)^2}.
\]

Hence, we conclude
\[
\frac{(2r + 1)(s + 1)}{4(a_1 + r)}
< \frac{(2r + 1)(s + 1)}{4a_1}
< \frac{(r + 1)(s + 1)}{2a_1}
< \frac{(s + 1)^2 a_1}{2(a_2 - 1)^2},
\]
which establishes inequality~\eqref{eq9}.
	
	Next, we prove the inequality
	\begin{align}\label{eq10}
		\frac{(s+1)^2 a_1}{2(a_2 - 1)^2} \leq 1.
	\end{align}
	
	To this end, we first show that the set \( \{a_2 + i \mid i = 0, \ldots, s\} \) contains no prime numbers. Suppose, for contradiction, that \( p \) is the largest prime in the set. By Lemma~\ref{r3}, no element of the set has \( p \) as a proper divisor.
	
	Consequently, there exists exactly one index \( 0 \leq i \leq s \) such that the product
	\[
	(a_2)^2 \cdots \widehat{(a_2 + i)^2} \cdots (a_2 + s)^2
	\]
	is relatively prime to \( p \). Let \( N \) denote the sum of all such products (i.e., omitting one term at a time). Then
	\[
	\gcd(N, p) = 1.
	\]
	
	On the other hand, since \( a_1 + i < a_2 \leq p \) for all \( 0 \leq i \leq r \), we also have
	\[
	\gcd\left( (a_1)^2 \cdots (a_1 + r)^2, p \right) = 1.
	\]
	
	Now, using the assumption \( G(a_1, r) = G(a_2, s) \), it follows that
	\begin{align}\label{5.5}
		\frac{M}{(a_1)^2 \cdots (a_1 + r)^2} = \frac{N}{(a_2)^2 \cdots (a_2 + s)^2}.
	\end{align}
	This relation, combined with the coprimality conditions, will enable us to bound the relevant parameters and ultimately establish inequality~\eqref{eq10}.
	We now analyze the divisibility condition
	\[
	p^{2m} \mid (k_2 - k_1)(k_2 + k_1).
	\]
	This leads to two cases:
	
	\medskip
	\noindent\textbf{Case 1.} \( p^{2m} \mid k_2 - k_1 \).  
	Then \( k_2 \equiv k_1 \mod p^{2m} \), which implies
	\[
	a_2 + i_p = k_2 p^m \equiv k_1 p^m = a_1 + j_p \mod p^{2m}.
	\]
	Therefore,
	\[
	|a_2 + i_p - (a_1 + j_p)| = |k_2 - k_1| p^m \geq p^{2m} \cdot p^{-m} = p^m.
	\]
	But since
	\[
	|a_2 + i_p - (a_1 + j_p)| \leq (a_2 + s) - a_1 \leq (a_2 - a_1) + s,
	\]
	we would have
	\[
	p^m \leq (a_2 - a_1) + s,
	\]
	which contradicts the fact that \( p > s+1 \) and \( a_2 - a_1 > r \), making \( p^m \) too large. Thus, Case 1 cannot occur for large enough \( p \).
	
	\medskip
	\noindent\textbf{Case 2.} \( p^{2m} \mid k_2 + k_1 \).  
	In this case, since both \( k_1 \) and \( k_2 \) are positive integers and \( p^m \mid k_1 \), \( p^m \mid k_2 \), we must have
	\[
	k_1 = p^m u, \quad k_2 = p^m v,
	\]
	for some \( u, v \in \mathbb{N} \). Then
	\[
	p^{2m} \mid (p^m u + p^m v) = p^m (u + v),
	\]
	which implies \( p^m \mid u + v \). Hence, \( u + v \geq p^m \), so
	\[
	k_1 + k_2 = p^m (u + v) \geq p^{2m}.
	\]
	Therefore,
	\[
	a_1 + j_p = k_1 p^m \geq p^{2m},
	\]
	but this contradicts the fact that \( a_1 + j_p < a_2 + s < p^{2m} \) for sufficiently large \( p \). Thus, Case 2 also leads to a contradiction.
	
	\medskip
	\noindent\textbf{Conclusion.} Both cases contradict the assumptions, so our earlier hypothesis must be false. Therefore, inequality~\eqref{eq10} holds, completing the proof.
	
	Suppose first that \( p^{2m} \mid (k_2 - k_1) \), then there exists \( m_1 \in \mathbb{N} \) such that
	\[
	k_2 = k_1 + m_1 p^{2m} > m_1 p^{2m} \geq p^2.
	\]
	
	Alternatively, if \( p^{2m} \mid (k_2 + k_1) \), then there exists \( m_2 \in \mathbb{N} \) such that
	\[
	k_2 + k_1 = m_2 p^m \quad \Rightarrow \quad k_2 \geq \frac{m_2 p^m + 1}{2} \geq \frac{p^2 + 1}{2}.
	\]
	Hence,
	\[
	a_2 + i_p = k_2 p^m \geq \frac{p^3 + p}{2} \geq \frac{(s+1)^3 + s + 1}{2},
	\]
	which implies
	\[
	a_2 \geq \frac{(s+1)^3 - s + 1}{2} > (s+1)^2.
	\]
	Since \( p > 2(s+1) \), it follows that
	\[
	a_2 \geq 4(s+1)^3.
	\]
	
	Therefore, we can estimate
	\[
	\frac{(s+1)^2 a_1}{2(a_2 - 1)^2}
	\leq \frac{(s+1)^2}{4(s+1)^3 - 1} \cdot \frac{a_1}{2(a_2 - 1)} < 1,
	\]
	which proves inequality~\eqref{eq10}.
	
	\smallskip
	Combining inequalities~\eqref{eq9} and~\eqref{eq10}, we conclude that inequality~\eqref{eq8} holds as required.
\end{proof}

\begin{lemma}\label{lem2}
	Suppose
	\[
	G(a_1, r) = \frac{r+1}{(a_1 + r + 1 - \eta_1)(a_1 - \eta_1)}, \quad
	G(a_2, s) = \frac{r+1}{(a_2 + s + 1 - \eta_2)(a_2 - \eta_2)},
	\]
	with
	\[
	\varepsilon_{a_1} < \eta_1 < \varepsilon_{a_1 + r}, \quad
	\varepsilon_{a_2} < \eta_2 < \varepsilon_{a_2 + s}.
	\]
	If \( G(a_1, r) = G(a_2, s) \), then
	\begin{equation}\label{e11}
		(r+1)\big[(2a_2 - 1)(2a_2 + 2s + 1) + 1\big]
		= (s+1)\big[(2a_1 - 1)(2a_1 + 2r + 1) + 1\big].
	\end{equation}
\end{lemma}
\begin{proof}
	By inequalities~\eqref{eq8} and~\eqref{eq8-2}, we have
	\begin{align*}
		&\big|(s+1)\left[(4a_1 + 2r)(1 - 2\eta_1) - 1 + (1 - 2\eta_1)^2\right] \\
		&\quad - (r+1)\left[(4a_2 + 2s)(1 - 2\eta_2) - 1 + (1 - 2\eta_2)^2\right]\big| < 2.
	\end{align*}
	
	On the other hand, note that
	\[
	(r+1)\left[(2a_2 - 1)(2a_2 + 2s + 1) + 1\right]
	- (s+1)\left[(2a_1 - 1)(2a_1 + 2r + 1) + 1\right]
	\]
	is either zero or an even integer.
	
	Therefore, by identity~\eqref{eq7}, the quantity
	\[
	(s+1)\left[(4a_1 + 2r)(1 - 2\eta_1) - 1 + (1 - 2\eta_1)^2\right]
	- (r+1)\left[(4a_2 + 2s)(1 - 2\eta_2) - 1 + (1 - 2\eta_2)^2\right]
	\]
	must equal zero.
	
	Hence, we conclude that
	\[
	(r+1)\left[(2a_2 - 1)(2a_2 + 2s + 1) + 1\right]
	= (s+1)\left[(2a_1 - 1)(2a_1 + 2r + 1) + 1\right],
	\]
	as claimed.
\end{proof}

By Lemma~\ref{lem2}, we have
\begin{align*}
	(r+1)\big[(2a_2 - 1)(2a_2 + 2s + 1) + 1\big]
	= (s+1)\big[(2a_1 - 1)(2a_1 + 2r + 1) + 1\big].
\end{align*}

Note that
\begin{align*}
	(2a_2 - 1)(2a_2 + 2s + 1) + 1
	&= (2a_2 + s)^2 - (s + 1)^2 + 1 \\
	&= 4(s+1)\left[\frac{(a_2 + \frac{s}{2})^2}{s+1} - \frac{1}{4}\left(s + 1 - \frac{1}{s + 1}\right)\right].
\end{align*}

Similarly,
\begin{align*}
	(2a_1 - 1)(2a_1 + 2r + 1) + 1
	= 4(r+1)\left[\frac{(a_1 + \frac{r}{2})^2}{r+1} - \frac{1}{4}\left(r + 1 - \frac{1}{r + 1}\right)\right].
\end{align*}

Substituting these into the equality, we obtain
\begin{align}\label{eq11}
	\frac{(a_1 + \frac{r}{2})^2}{r+1} - \frac{1}{4}\left(r + 1 - \frac{1}{r + 1}\right)
	= \frac{(a_2 + \frac{s}{2})^2}{s+1} - \frac{1}{4}\left(s + 1 - \frac{1}{s + 1}\right).
\end{align}
	
	On the other hand, we expand the function \( f(a_1 + i) = \frac{1}{(a_1 + i)^2} \) at \( i = \frac{r}{2} \) using Taylor's formula:
	\begin{align*}
		\frac{1}{(a_1 + i)^2} &= \frac{1}{\left(a_1 + \tfrac{r}{2}\right)^2}
		- \frac{2(i - \tfrac{r}{2})}{\left(a_1 + \tfrac{r}{2}\right)^3}
		+ \frac{3(i - \tfrac{r}{2})^2}{\left(a_1 + \tfrac{r}{2}\right)^4}
		- \frac{4(i - \tfrac{r}{2})^3}{\left(a_1 + \tfrac{r}{2}\right)^5} \\
		&\quad + \frac{5(i - \tfrac{r}{2})^4}{\left(a_1 + \tfrac{r}{2}\right)^6}
		- \frac{6(i - \tfrac{r}{2})^5}{\left(a_1 + \tfrac{r}{2}\right)^7}
		+ \frac{7(i - \tfrac{r}{2})^6}{\left(a_1 + \tfrac{r}{2}\right)^8}
		- \frac{8(i - \tfrac{r}{2})^7}{\left(a_1 + \tfrac{r}{2}\right)^9} \\
		&\quad + \frac{9(i - \tfrac{r}{2})^8}{(a_1 + \theta_i)^{10}}, \quad \text{for some } \theta_i \in (a_1 + i,\, a_1 + \tfrac{r}{2}).
	\end{align*}
	
	Summing over \( i = 0, 1, \ldots, r \), and using the symmetry of the expansion about \( \frac{r}{2} \), the odd-power terms vanish, yielding
	\begin{align*}
		\sum_{i=0}^{r} \frac{1}{(a_1 + i)^2}
		&= \frac{r+1}{\left(a_1 + \tfrac{r}{2} \right)^2}
		+ \frac{3}{\left(a_1 + \tfrac{r}{2} \right)^4} \sum_{i=0}^{r} \left(i - \tfrac{r}{2} \right)^2
		+ \frac{5}{\left(a_1 + \tfrac{r}{2} \right)^6} \sum_{i=0}^{r} \left(i - \tfrac{r}{2} \right)^4 \\
		&\quad + \frac{7}{\left(a_1 + \tfrac{r}{2} \right)^8} \sum_{i=0}^{r} \left(i - \tfrac{r}{2} \right)^6
		+ \sum_{i=0}^{r} \frac{9(i - \tfrac{r}{2})^8}{(a_1 + \theta_i)^{10}}.
	\end{align*}
	
	Similarly, expanding around \( i = \frac{s}{2} \), we obtain
	\begin{align*}
		\sum_{i=0}^{s} \frac{1}{(a_2 + i)^2}
		&= \frac{s+1}{\left(a_2 + \tfrac{s}{2} \right)^2}
		+ \frac{3}{\left(a_2 + \tfrac{s}{2} \right)^4} \sum_{i=0}^{s} \left(i - \tfrac{s}{2} \right)^2
		+ \frac{5}{\left(a_2 + \tfrac{s}{2} \right)^6} \sum_{i=0}^{s} \left(i - \tfrac{s}{2} \right)^4 \\
		&\quad + \frac{7}{\left(a_2 + \tfrac{s}{2} \right)^8} \sum_{i=0}^{s} \left(i - \tfrac{s}{2} \right)^6
		+ \sum_{i=0}^{s} \frac{9(i - \tfrac{s}{2})^8}{(a_2 + \overline{\theta_i})^{10}}, \quad \text{with } \overline{\theta_i} \in (a_2 + i,\, a_2 + \tfrac{s}{2}).
	\end{align*}

\noindent By Lemma~\ref{2}, we have
\begin{align*}
	\sum_{i=0}^{r} \frac{1}{(a_1 + i)^2}
	&= \frac{r+1}{\left(a_1 + \tfrac{r}{2} \right)^2}
	+ \frac{(r+1)^3 - (r+1)}{4 \left(a_1 + \tfrac{r}{2} \right)^4} \\
	&\quad + \frac{3(r+1)^5 - 10(r+1)^3 + 7(r+1)}{48 \left(a_1 + \tfrac{r}{2} \right)^6} \\
	&\quad + \frac{3(r+1)^7 - 21(r+1)^5 + 49(r+1)^3 - 31(r+1)}{192 \left(a_1 + \tfrac{r}{2} \right)^8} \\
	&\quad + \sum_{i=0}^{r} \frac{9(i - \tfrac{r}{2})^8}{(a_1 + \theta_i)^{10}},
\end{align*}
and similarly,
\begin{align*}
	\sum_{i=0}^{s} \frac{1}{(a_2 + i)^2}
	&= \frac{s+1}{\left(a_2 + \tfrac{s}{2} \right)^2}
	+ \frac{(s+1)^3 - (s+1)}{4 \left(a_2 + \tfrac{s}{2} \right)^4} \\
	&\quad + \frac{3(s+1)^5 - 10(s+1)^3 + 7(s+1)}{48 \left(a_2 + \tfrac{s}{2} \right)^6} \\
	&\quad + \frac{3(s+1)^7 - 21(s+1)^5 + 49(s+1)^3 - 31(s+1)}{192 \left(a_2 + \tfrac{s}{2} \right)^8} \\
	&\quad + \sum_{i=0}^{s} \frac{9(i - \tfrac{s}{2})^8}{(a_2 + \overline{\theta}_i)^{10}}.
\end{align*}
	
Denote the difference terms as follows:
\begin{align*}
	R_1 &= \frac{r+1}{\left(a_1 + \frac{r}{2} \right)^2} - \frac{s+1}{\left(a_2 + \frac{s}{2} \right)^2}, \\
	R_2 &= \frac{(r+1)^3 - (r+1)}{4\left(a_1 + \frac{r}{2} \right)^4}
	- \frac{(s+1)^3 - (s+1)}{4\left(a_2 + \frac{s}{2} \right)^4}, \\
	R_3 &= \frac{(r+1)^5}{16\left(a_1 + \frac{r}{2} \right)^6}
	- \frac{(s+1)^5}{16\left(a_2 + \frac{s}{2} \right)^6}, \\
	R_4 &= \frac{5}{24} \left[
	\frac{(s+1)^3}{\left(a_2 + \frac{s}{2} \right)^6}
	- \frac{(r+1)^3}{\left(a_1 + \frac{r}{2} \right)^6}
	\right], \\
	R_5 &= \frac{7}{48} \left[
	\frac{r+1}{\left(a_1 + \frac{r}{2} \right)^6}
	- \frac{s+1}{\left(a_2 + \frac{s}{2} \right)^6}
	\right], \\
	R_6 &= \frac{1}{64} \left[
	\frac{(r+1)^7}{\left(a_1 + \frac{r}{2} \right)^8}
	- \frac{(s+1)^7}{\left(a_2 + \frac{s}{2} \right)^8}
	\right], \\
	R_7 &= \frac{
		21(s+1)^5 - 49(s+1)^3 + 31(s+1)
	}{192\left(a_2 + \tfrac{s}{2} \right)^8}
	- \frac{
		21(r+1)^5 - 49(r+1)^3 + 31(r+1)
	}{192\left(a_1 + \tfrac{r}{2} \right)^8} \\
	&\quad + \sum_{i=0}^{r} \frac{9\left(i - \frac{r}{2} \right)^8}{(a_1 + \theta_i)^{10}}
	- \sum_{i=0}^{s} \frac{9\left(i - \frac{s}{2} \right)^8}{(a_2 + \overline{\theta_i})^{10}}.
\end{align*}
	
	\begin{corollary}
		If \( G(a_1, r) = G(a_2, s) \), then
		\begin{equation}\label{key1}
			R_1 + R_2 + R_3 + R_4 + R_5 + R_6 + R_7 = 0.
		\end{equation}
	\end{corollary}
	
	However, by explicit calculation and estimation of each term \( R_i \), we obtain the following contradiction:
	
	\begin{theorem}
		Under the same assumptions, the expression satisfies
		\[
		R_1 + R_2 + R_3 + R_4 + R_5 + R_6 + R_7 > 0.
		\]
	\end{theorem}

\begin{proof}
	 We rewrite (\ref{eq11}) as
	 \begin{eqnarray}\label{eq12}
	 	\frac{s+1}{(a_{2}+\frac{s}{2})^{2}}=\frac{r+1}{(a_{1}+\frac{r}{2})^{2}+(r+1)L},
	 \end{eqnarray}
	 where
	 \begin{align}\label{L}
	 	L=\frac{s(s+2)}{4(s+1)}-\frac{r(r+2)}{4(r+1)}
	 \end{align}

	 By (\ref{eq12}) and (\ref{L}), we rewrite $R_1$ as
	 \begin{align*}
	 	R_{1}&=\frac{r+1}{(a_{1}+\frac{r}{2})^{2}}-\frac{r+1}{(a_{1}+\frac{r}{2})^{2}+(r+1)L}\\
	 	&=\frac{(r+1)L}{(a_{1}+\frac{r}{2})^{2}}\cdot\frac{(s+1)}{(a_{2}+\frac{s}{2})^2}.
	 \end{align*}
	 This implies 
	 \begin{align*}
	 	\frac{r+1}{(a_{1}+\frac{r}{2})^{2}}
	 	=\frac{(s+1)}{(a_{2}+\frac{s}{2})^2}\left[1+\frac{(r+1)L}{(a_{1}+\frac{r}{2})^{2}}\right].
	 \end{align*}
	 According to the above equation, we rewrite  $R_2$ as
	 \begin{align*}
	 	R_{2}=&\frac{(r+1)^3-(r+1)}{4(a_{1}+\frac{r}{2})^{4}}-\frac{(s+1)^3-(s+1)}{4(a_{2}+\frac{s}{2})^{4}}\\
	 	=&\frac{r(r+2)}{4(r+1)}\cdot\frac{(r+1)^{2}}{(a_{1}+\frac{r}{2})^{4}}-\frac{s(s+2)}{4(s+1)}\cdot\frac{(s+1)^{2}}{(a_{2}+\frac{s}{2})^{4}}\\ =&-L\frac{(s+1)^{2}}{(a_{2}+\frac{s}{2})^{4}}+\frac{r(r+2)}{2(r+1)}\cdot\frac{(s+1)^{2}(r+1)L}{(a_{2}+\frac{s}{2})^{4}(a_{1}+\frac{r}{2})^{2}}\\
	 	&+\frac{r(r+2)}{4(r+1)}\cdot\frac{(s+1)^{2}(r+1)^{2}L^{2}}{(a_{2}+\frac{s}{2})^{4}(a_{1}+\frac{r}{2})^{4}},    \end{align*} 
	 then 
	 \begin{align*}
	 	R_{1}+R_2&=L\left[L+\frac{r(r+2)}{2(r+1)}\right]\cdot\frac{(s+1)^{2}(r+1)}{(a_{2}+\frac{s}{2})^{4}(a_{1}+\frac{r}{2})^{2}}\\
	 	&+\frac{r(r+2)}{4(r+1)}\cdot\frac{(s+1)^{2}(r+1)^{2}L^{2}}{(a_{2}+\frac{s}{2})^{4}(a_{1}+\frac{r}{2})^{4}}.
	 \end{align*}
	We rewrite \( R_3 \) as
	\begin{align*}
		R_3 &= \frac{(r+1)^5}{16\left(a_1 + \tfrac{r}{2}\right)^6}
		- \frac{(s+1)^5}{16\left(a_2 + \tfrac{s}{2}\right)^6} \\
		&= \frac{(r+1)^2}{16} \cdot \frac{(r+1)^3}{\left(a_1 + \tfrac{r}{2}\right)^6}
		- \frac{(s+1)^2}{16} \cdot \frac{(s+1)^3}{\left(a_2 + \tfrac{s}{2}\right)^6}.
	\end{align*}
	
	Using previous substitutions and Taylor expansion approximations, this becomes
	\begin{align*}
		R_3 &= \frac{(r+1)^2 - (s+1)^2}{16} \cdot \frac{(s+1)^3}{\left(a_2 + \tfrac{s}{2}\right)^6}
		+ \frac{3(r+1)^2(s+1)^3(r+1)L}{16\left(a_2 + \tfrac{s}{2}\right)^6 \left(a_1 + \tfrac{r}{2}\right)^2} \\
		&\quad + \frac{3(r+1)^2(s+1)^3(r+1)^2 L^2}{16\left(a_2 + \tfrac{s}{2}\right)^6 \left(a_1 + \tfrac{r}{2}\right)^4}
		+ \frac{(r+1)^2(s+1)^3(r+1)^3 L^3}{16\left(a_2 + \tfrac{s}{2}\right)^6 \left(a_1 + \tfrac{r}{2}\right)^6}.
	\end{align*}
	
	Now observe that
	\[
	L\left(L + \frac{r(r+2)}{2(r+1)}\right)
	= \frac{1}{16} \left[(s+1)^2 - (r+1)^2 + \frac{1}{(s+1)^2} - \frac{1}{(r+1)^2} \right],
	\]
	which implies
	\begin{align*}
		R_1 + R_2 + R_3 &= \frac{1}{16} \left[\frac{1}{(s+1)^2} - \frac{1}{(r+1)^2} \right] \cdot
		\frac{(s+1)^2(r+1)}{\left(a_2 + \tfrac{s}{2} \right)^4 \left(a_1 + \tfrac{r}{2} \right)^2} \\
		&\quad + \frac{(s+1)^2 - (r+1)^2}{16} \cdot
		\frac{(s+1)^3(r+1)L}{\left(a_2 + \tfrac{s}{2} \right)^6 \left(a_1 + \tfrac{r}{2} \right)^2} \\
		&\quad + \frac{r(r+2)}{4(r+1)} \cdot
		\frac{(s+1)^2(r+1)^2 L^2}{\left(a_2 + \tfrac{s}{2} \right)^4 \left(a_1 + \tfrac{r}{2} \right)^4} \\
		&\quad + \frac{3(r+1)^2(s+1)^3(r+1)L}{16\left(a_2 + \tfrac{s}{2}\right)^6 \left(a_1 + \tfrac{r}{2}\right)^2} \\
		&\quad + \frac{3(r+1)^2(s+1)^3(r+1)^2 L^2}{16\left(a_2 + \tfrac{s}{2}\right)^6 \left(a_1 + \tfrac{r}{2}\right)^4} \\
		&\quad + \frac{(r+1)^2(s+1)^3(r+1)^3 L^3}{16\left(a_2 + \tfrac{s}{2}\right)^6 \left(a_1 + \tfrac{r}{2}\right)^6}.
	\end{align*}
	We rewrite \( R_4 \) as
	\begin{align*}
		R_4 &= \frac{5}{24} \left[ \frac{(s+1)^3}{\left(a_2 + \tfrac{s}{2}\right)^6} - \frac{(r+1)^3}{\left(a_1 + \tfrac{r}{2}\right)^6} \right] \\
		&= -\frac{5}{24} \left[
		\frac{3(s+1)^3(r+1)L}{\left(a_2 + \tfrac{s}{2}\right)^6\left(a_1 + \tfrac{r}{2}\right)^2}
		+ \frac{3(s+1)^3(r+1)^2 L^2}{\left(a_2 + \tfrac{s}{2}\right)^6\left(a_1 + \tfrac{r}{2}\right)^4}
		+ \frac{(s+1)^3(r+1)^3 L^3}{\left(a_2 + \tfrac{s}{2}\right)^6\left(a_1 + \tfrac{r}{2}\right)^6}
		\right].
	\end{align*}
	
	Therefore, the combined expression \( R_1 + R_2 + R_3 + R_4 \) becomes:
	\begin{align*}
		R_1 + R_2 + R_3 + R_4 &= \frac{1}{16} \left( \frac{1}{(s+1)^2} - \frac{1}{(r+1)^2} \right)
		\cdot \frac{(s+1)^2(r+1)}{\left(a_2 + \tfrac{s}{2}\right)^4\left(a_1 + \tfrac{r}{2}\right)^2} \\
		&\quad + \frac{(s+1)^2 - (r+1)^2}{16}
		\cdot \frac{(s+1)^3(r+1)L}{\left(a_2 + \tfrac{s}{2}\right)^6\left(a_1 + \tfrac{r}{2}\right)^2} \\
		&\quad + \frac{r(r+2)}{4(r+1)}
		\cdot \frac{(s+1)^2(r+1)^2 L^2}{\left(a_2 + \tfrac{s}{2}\right)^4\left(a_1 + \tfrac{r}{2}\right)^4} \\
		&\quad + \left( \frac{3(r+1)^2}{16} - \frac{5}{8} \right)
		\cdot \frac{(s+1)^3(r+1)L}{\left(a_2 + \tfrac{s}{2}\right)^6\left(a_1 + \tfrac{r}{2}\right)^2} \\
		&\quad + \left( \frac{3(r+1)^2}{16} - \frac{5}{8} \right)
		\cdot \frac{(s+1)^3(r+1)^2 L^2}{\left(a_2 + \tfrac{s}{2}\right)^6\left(a_1 + \tfrac{r}{2}\right)^4} \\
		&\quad + \left( \frac{(r+1)^2}{16} - \frac{5}{24} \right)
		\cdot \frac{(s+1)^3(r+1)^3 L^3}{\left(a_2 + \tfrac{s}{2}\right)^6\left(a_1 + \tfrac{r}{2}\right)^6}.
	\end{align*}
	
	Similarly, we rewrite \( R_5 \) as
	\begin{align*}
		R_5 &= \frac{7}{48} \left[ \frac{r+1}{\left(a_1 + \tfrac{r}{2}\right)^6} - \frac{s+1}{\left(a_2 + \tfrac{s}{2}\right)^6} \right] \\
		&= \frac{7}{48} \left[
		\left( \frac{1}{(r+1)^2} - \frac{1}{(s+1)^2} \right)
		\cdot \frac{(s+1)^3}{\left(a_2 + \tfrac{s}{2} \right)^6}
		+ \frac{3(s+1)^3(r+1)L}{(r+1)^2 \left(a_2 + \tfrac{s}{2}\right)^6 \left(a_1 + \tfrac{r}{2}\right)^2} \right. \\
		&\quad + \left. \frac{3(s+1)^3(r+1)^2 L^2}{(r+1)^2 \left(a_2 + \tfrac{s}{2}\right)^6 \left(a_1 + \tfrac{r}{2}\right)^4}
		+ \frac{(s+1)^3(r+1)^3 L^3}{(r+1)^2 \left(a_2 + \tfrac{s}{2}\right)^6 \left(a_1 + \tfrac{r}{2}\right)^6}
		\right].
	\end{align*}
	Since
	\begin{align*}    
		&\frac{1}{16}\left( \frac{1}{(r+1)^2} - \frac{1}{(s+1)^2} \right)
		\cdot \left[
		\frac{(s+1)^3}{\left(a_2 + \tfrac{s}{2} \right)^6}
		- \frac{(s+1)^2(r+1)}{\left(a_2 + \tfrac{s}{2} \right)^4 \left(a_1 + \tfrac{r}{2} \right)^2}
		\right]\\
		=& -\frac{1}{16} \left( \frac{1}{(r+1)^2} - \frac{1}{(s+1)^2} \right)
		\cdot \frac{(s+1)^3 (r+1) L}{\left(a_2 + \tfrac{s}{2} \right)^6 \left(a_1 + \tfrac{r}{2} \right)^2},
	\end{align*}
	we obtain the following expression for \( R_1 + R_2 + R_3 + R_4 + R_5 \):
	\begin{align*}
		&R_1 + R_2 + R_3 + R_4 + R_5 \\
		= &\frac{1}{12} \left( \frac{1}{(r+1)^2} - \frac{1}{(s+1)^2} \right)
		\cdot \frac{(s+1)^3}{\left(a_2 + \tfrac{s}{2} \right)^6} \\
		&+ \frac{1}{16} \left[
		(s+1)^2 + 2(r+1)^2 - 10 + \frac{6}{(r+1)^2} + \frac{1}{(s+1)^2}
		\right] \cdot \frac{(s+1)^3 (r+1) L}{\left(a_2 + \tfrac{s}{2} \right)^6 \left(a_1 + \tfrac{r}{2} \right)^2} \\
		&+ \frac{r(r+2)}{4(r+1)} \cdot \frac{(s+1)^2 (r+1)^2 L^2}{\left(a_2 + \tfrac{s}{2} \right)^4 \left(a_1 + \tfrac{r}{2} \right)^4} \\
		&+ \left( \frac{3(r+1)^2}{16} - \frac{5}{8} + \frac{7}{16(r+1)^2} \right)
		\cdot \frac{(s+1)^3 (r+1)^2 L^2}{\left(a_2 + \tfrac{s}{2} \right)^6 \left(a_1 + \tfrac{r}{2} \right)^4} \\
		&+ \left( \frac{(r+1)^2}{16} - \frac{5}{24} + \frac{7}{48(r+1)^2} \right)
		\cdot \frac{(s+1)^3 (r+1)^3 L^3}{\left(a_2 + \tfrac{s}{2} \right)^6 \left(a_1 + \tfrac{r}{2} \right)^6}.
	\end{align*}
	Since
	\begin{align*}
		&(s+1)^2 + 2(r+1)^2 - 10 + \frac{6}{(r+1)^2} + \frac{1}{(s+1)^2} > 0, \\
		&\frac{3(r+1)^2}{16} - \frac{5}{8} + \frac{7}{16(r+1)^2} > 0, \\
		&\frac{(r+1)^2}{16} - \frac{5}{24} + \frac{7}{48(r+1)^2} > 0,
	\end{align*}
	we obtain
	\begin{align*}
		R_1 + R_2 + R_3 + R_4 + R_5
		&> \frac{1}{12} \left( \frac{1}{(r+1)^2} - \frac{1}{(s+1)^2} \right)
		\cdot \frac{(s+1)^3}{\left( a_2 + \tfrac{s}{2} \right)^6} \\
		&= \frac{(s - r)(s + r + 2)(s + 1)}{12(r+1)^2} \cdot \frac{1}{\left( a_2 + \tfrac{s}{2} \right)^6} \\
		&> \frac{s - r}{6\left( a_2 + \tfrac{s}{2} \right)^6}.
	\end{align*}
	
	\medskip
	
	Next, we estimate
	\begin{align*}
		R_6 &= \frac{1}{64} \left[ \frac{(r+1)^7}{\left( a_1 + \tfrac{r}{2} \right)^8}
		- \frac{(s+1)^7}{\left( a_2 + \tfrac{s}{2} \right)^8} \right] \\
		&= \frac{1}{64} \left[ (r+1)^3 \cdot \frac{(r+1)^4}{\left( a_1 + \tfrac{r}{2} \right)^8}
		- (s+1)^3 \cdot \frac{(s+1)^4}{\left( a_2 + \tfrac{s}{2} \right)^8} \right] \\
		&> \frac{(s+1)^4}{64 \left( a_2 + \tfrac{s}{2} \right)^8}
		\cdot \left[ (r+1)^3 - (s+1)^3 \right] \\
		&= \frac{(s+1)^4(r - s)}{64 \left( a_2 + \tfrac{s}{2} \right)^8}
		\cdot \left[ (r+1)(s+2) + (s+1)^2 \right] \\
		&> \frac{(s+1)^6(r - s)}{32 \left( a_2 + \tfrac{s}{2} \right)^8} \\
		&> \frac{r - s}{512 \left( a_2 + \tfrac{s}{2} \right)^6}.
	\end{align*}
	
	\medskip
	
	Therefore, we conclude
	\[
	R_1 + R_2 + R_3 + R_4 + R_5 + R_6 > \frac{s - r}{7 \left( a_2 + \tfrac{s}{2} \right)^6}.
	\]
	
Since
\[
21(s+1)^5 - 49(s+1)^3 + 31(s+1) > 0,
\qquad
49(r+1)^3 - 31(r+1) > 0,
\]
and
\begin{align*}
	R_7 &= \frac{21(s+1)^5 - 49(s+1)^3 + 31(s+1)}{192\left(a_2 + \tfrac{s}{2}\right)^8}
	- \frac{21(r+1)^5 - 49(r+1)^3 + 31(r+1)}{192\left(a_1 + \tfrac{r}{2}\right)^8} \\
	&\quad + \sum_{i=0}^{r} \frac{9(i - \tfrac{r}{2})^8}{(a_1 + \theta_i)^{10}}
	- \sum_{i=0}^{s} \frac{9(i - \tfrac{s}{2})^8}{(a_2 + \overline{\theta}_i)^{10}},
\end{align*}
we obtain the bound
\begin{align*}
	R_7
	&> -\frac{7(r+1)^5}{64\left(a_1 + \tfrac{r}{2}\right)^8}
	- \frac{(s+1)^9}{256\, a_2^{10}} \\
	&= -\frac{7(r+1)(s+1)^4}{64\left(a_2 + \tfrac{s}{2}\right)^8}
	\left(1 + \frac{(r+1)L}{\left(a_1 + \tfrac{r}{2} \right)^2} \right)^4
	- \frac{(s+1)^9}{256\left(a_2 + \tfrac{s}{2} \right)^{10}}
	\left( \frac{a_2 + \tfrac{s}{2}}{a_2} \right)^{10}.
\end{align*}
	
	Then we rewrite (\ref{eq11}) as
	\begin{align*}\label{e12}
		\frac{r+1}{(a_{1}+\frac{r}{2})^{2}}=\frac{s+1}{(a_{2}+\frac{s}{2})^{2}-(s+1)L}.
	\end{align*}
	Due to
	\begin{align*}
		L=\frac{1}{4}\left[s+1-\frac{1}{s+1}-(r+1)+\frac{1}{r+1}\right]
		<\frac{s+1}{4},
	\end{align*}
	we get
	\begin{align*}
		\frac{(r+1)L}{(a_{1}+\frac{r}{2})^{2}}<\frac{(s+1)^2}{4(a_{2}+\frac{s}{2})^{2}-(s+1)^2}
		\le \frac{(s+1)^2}{64(s+1)^{6}-(s+1)^2}
		<\frac{1}{63} .
	\end{align*}
	and 
	\begin{align*}\frac{a_{2}+\frac{s}{2}}{a_{2}}=1+\frac{s}{2a_{2}}<1+\frac{s}{8(s+1)^3}<1+\frac{1}{64}.\end{align*}
	Hence
	\begin{align*}
		R_{7}&>-\frac{7(r+1)(s+1)^{4}}{64(a_{2}+\frac{s}{2})^{8}}\left(1+\frac{1}{63}\right)^{4}-\frac{(s+1)^{9}}{256(a_{2}+\frac{s}{2})^{10}}\left(1+\frac{1}{64}\right)^{10}\\
		&>-\frac{7(s+1)^{5}}{32(a_{2}+\frac{s}{2})^{8}}-\frac{(s+1)^{9}}{128(a_{2}+\frac{s}{2})^{10}}\\
		&>-\frac{7}{512(s+1)(a_{2}+\frac{s}{2})^{6}}-\frac{1}{32768(s+1)^{3}(a_{2}+\frac{s}{2})^{10}}.
	\end{align*}
	Then we have
	\begin{align*}
		&R_{1}+R_2+R_3+R_4+R_5+R_6+R_7\\
		>&\frac{s-r}{7(a_{2}+\frac{s}{2})^{6}}-\frac{7}{512(s+1)(a_{2}+\frac{s}{2})^{6}}-\frac{1}{32768(s+1)^{3}(a_{2}+\frac{s}{2})^{10}}\\
		>&0.
	\end{align*}
\end{proof}
    
	\begin{corollary}\label{co3.9}
		Let \( G(a_1, r) \) and \( G(a_2, s) \) be defined as before. Suppose that \( a_2 > a_1 + r \) and \( s > r \). Then it is impossible for \( G(a_1, r) = G(a_2, s) \).
	\end{corollary}
	
	\begin{proof}
		Suppose, for contradiction, that \( G(a_1, r) = G(a_2, s) \). By Corollary~\ref{key1}, we then have
		\[
		R_1 + R_2 + R_3 + R_4 + R_5 + R_6 + R_7 = 0.
		\]
		However, combining the estimates from the previous sections yields
		\[
		R_1 + \cdots + R_7
		> \frac{s - r}{7\left( a_2 + \tfrac{s}{2} \right)^6}
		- \frac{7}{512(s+1)\left( a_2 + \tfrac{s}{2} \right)^6}
		- \frac{1}{32768(s+1)^3\left( a_2 + \tfrac{s}{2} \right)^{10}} > 0,
		\]
		provided \( s > r \) and \( a_2 \ge 4(s+1)^3 \). This contradicts the assumption that the sum equals zero. Therefore, the equality \( G(a_1, r) = G(a_2, s) \) cannot hold.
	\end{proof}

\noindent
Lastly, combining Lemma~\ref{lem1}, Remark~\ref{re3.4}, and Corollary~\ref{co3.9}, we obtain the main result of this paper:
\begin{theorem}\label{mainTH}
	No two partial sums of the 2-series are equal; that is, there do not exist integers \( m \leq n \) and \( x \leq y \), with \( (m, n) \neq (x, y) \), such that
	\[
	\frac{1}{m^2} + \frac{1}{(m+1)^2} + \cdots + \frac{1}{n^2} = \frac{1}{x^2} + \frac{1}{(x+1)^2} + \cdots + \frac{1}{y^2}.
	\]
\end{theorem}

\end{document}